\documentclass[reqno]{amsart}
\usepackage{amsmath, amssymb, amsthm, epsfig}
\usepackage{hyperref, latexsym}
\usepackage{url}
\usepackage{color}

\newcommand{\intav}[1]{\mathchoice {\mathop{\vrule width 6pt height 3 pt depth  -2.5pt
\kern -8pt \intop}\nolimits_{\kern -6pt#1}} {\mathop{\vrule width
5pt height 3  pt depth -2.6pt \kern -6pt \intop}\nolimits_{#1}}
{\mathop{\vrule width 5pt height 3 pt depth -2.6pt \kern -6pt
\intop}\nolimits_{#1}} {\mathop{\vrule width 5pt height 3 pt depth
-2.6pt \kern -6pt \intop}\nolimits_{#1}}}

\def\today{\ifcase\month\or
  January\or February\or March\or April\or May\or June\or
  July\or August\or September\or October\or November\or December\fi
  \space\number\day, \number\year}

\DeclareMathOperator{\supp}{\mathrm{supp}}

 \newtheorem{theorem}{Theorem}
 \newtheorem{lemma}[theorem]{Lemma}

 \theoremstyle{definition}

 \theoremstyle{remark}
 
 \newcommand{\mc}{\mathcal}

 \newcommand{\C}{\mathbb{C}}
 \newcommand{\R}{\mathbb{R}}
 \newcommand{\N}{\mathbb{N}}
 
 \newcommand{\Z}{\mathbb{Z}}

 \newcommand{\dx}{\text{\rm d}x}
 
 \newcommand{\dy}{\text{\rm d}y}

\newcommand{\dxi}{\text{\rm d}\xi}

\begin{document}

\title[On the variation of maximal operators of convolution type]{On the variation of maximal operators of convolution type}
\author[E. Carneiro and B. F. Svaiter]{Emanuel Carneiro* and Benar F. Svaiter}
\date{\today}
\thanks{*corresponding author email: carneiro@impa.br}
\address{IMPA - Estrada Dona Castorina, 110, Rio de Janeiro, RJ, Brazil 22460-320}
\email{carneiro@impa.br}
\email{benar@impa.br}
\allowdisplaybreaks
\numberwithin{equation}{section}

\maketitle

\begin{abstract}
In this paper we study the regularity properties of two maximal operators of convolution type: the heat flow maximal operator (associated to the Gauss kernel) and the Poisson maximal operator (associated to the Poisson kernel). In dimension $d=1$ we prove that these maximal operators do not increase  the $L^p$-variation of a function for any $p \geq 1$, while in dimensions $d>1$ we obtain the corresponding results for the $L^2$-variation. Similar results are proved for the discrete versions of these operators.
\end{abstract}

\section{Introduction}
\subsection{Background} Let $\varphi \in L^1(\R^d)$ be a nonnegative function such that 
$$\int_{\R^d} \varphi(x)\,\dx =1.$$
We let $\varphi_t(x) = t^{-d} \varphi(t^{-1}x)$ and consider the maximal operator $\mc{M}_{\varphi}$ associated to this approximation of the identity
\begin{equation}\label{Conv_type_max}
\mc{M}_{\varphi}f(x) = \sup_{t>0}\big(|f|*\varphi_t\big)(x).
\end{equation}
The Hardy-Littlewood maximal function, henceforth denoted by $M$, occurs when we consider $\varphi(x) = (1/m(B_1)) \chi_{B_1}(x)$, where $B_1$ is the $d$-dimensional ball centered at the origin with radius $1$ and $m(B_1)$ is its Lebesgue measure. In a certain sense, one could say that $M$ controls other such maximal operators of convolution type. In fact, if our $\varphi$ admits a radial non-increasing majorant in $L^1(\R^d)$ with integral $A$, from \cite[Chapter III, Theorem 2]{S} we know that
\begin{equation}\label{orig_bound}
\mc{M}_{\varphi}f(x) \leq A \, Mf(x)
\end{equation}
for all $x \in \R^d$ and thus we obtain the boundedness of $\mc{M}_{\varphi}$ from  $L^p(\R^d)$ to $L^p(\R^d)$ if $p>1$, and from $L^1(\R^d)$ to $L^{1,\infty}(\R^d)$ in the case $p=1$.

\medskip

Over the last years there has been considerable effort in understanding the effects of the Hardy-Littlewood maximal operator $M$, and some of its variants, in Sobolev functions. In \cite{Ki} Kinnunen showed that $M:W^{1,p}(\R^d) \to W^{1,p}(\R^d)$ is bounded for $p>1$. The paradigm that an $L^p$-bound implies a $W^{1,p}$-bound was later extended to a local version of $M$ in \cite{KL}, to a fractional version in \cite{KiSa} and to a multilinear version in \cite{CM}. The continuity of $M:W^{1, p}(\R^d) \to W^{1, p}(\R^d)$ for $p>1$ was established by Luiro in \cite{Lu1}. When $p=1$ the issues become more subtle. The question on whether the operator $f\mapsto \nabla M f$ is bounded from $W^{1,1}(\R^d)$ to $L^1(\R^d)$, posed by Haj\l asz and Onninen in \cite{HO}, remains open in its general case (see also \cite{HM}). Partial progress was achieved in the discrete setting in the work \cite{BCHP} for dimension $d=1$ and in the work \cite{CH} for general dimension $d>1$. In the continuous setting the only progress has been in dimension $d=1$. For the right (or left) Hardy-Littlewood maximal operator, which we call here $M_r$ (corresponding to $\varphi(x) = \chi_{[0,1]}(x)$ in \eqref{Conv_type_max}) Tanaka \cite{Ta} was the first to observe that, if $f \in W^{1,1}(\R)$, then $M_rf$ has a weak derivative and 
\begin{equation}\label{Intro2.1}
\|(M_rf)'\|_{L^1(\R)} \leq \|f'\|_{L^1(\R)},
\end{equation}
which led to the bound for the {\it non-centered} Hardy-Littlewood maximal operator $\widetilde{M}$,
\begin{equation}\label{Intro2}
\|(\widetilde{M}f)'\|_{L^1(\R)} \leq 2 \|f'\|_{L^1(\R)}.
\end{equation}
This was later refined by Aldaz and Perez-L\'{a}zaro \cite{AP} who obtained, under the assumption that $f$ is of bounded variation on $\R$, that $\widetilde{M}f$ is absolutely continuous and
\begin{equation}\label{Intro3.1}
V(\widetilde{M}f) \leq  V(f),
\end{equation}
where $V(f)$ denotes here the total variation of $f$. More recently, in the remarkable work \cite{Ku}, Kurka showed that if $f$ is of bounded variation on $\R$, then
\begin{equation}\label{Intro3}
V(Mf) \leq C \,V(f),
\end{equation}
for a certain $C>1$ (see \cite{Te} for the discrete analogue).

\medskip

In this paper we turn our attention to understanding the action of a general maximal operator of convolution type \eqref{Conv_type_max} in a Sobolev function. One can show that the original argument of Kinunnen \cite{Ki}, that proves that $M: W^{1,p}(\R^d) \to W^{1,p}(\R^d)$ is a bounded operator for $p>1$, can be applied to a general $\mc{M}_{\varphi}$ of type \eqref{Conv_type_max} that admits a radial non-increasing integrable majorant. In this case, if $f \in W^{1,p}(\R^d)$ with $p>1$, we have that $\nabla \mc{M}_{\varphi}f$ exists in the weak sense and
\begin{equation}\label{Intro4}
\|\nabla \mc{M}_{\varphi}f \|_p \leq C\, \|\nabla f \|_p,
\end{equation}
for a certain constant $C>1$. Therefore, from  \eqref{orig_bound} and \eqref{Intro4}, we already know that $\mc{M}_{\varphi}:W^{1, p}(\R^d) \to W^{1, p}(\R^d)$ is bounded for $p>1$. 

\medskip

We want to explore here the question on whether a maximal operator of convolution type can increase the variation (or $L^p$-variation) of a function. In other words, can we prove inequalities like \eqref{Intro3} (for $d=1$) and \eqref{Intro4} (for $p\geq 1$) with the constant $C=1$? To our knowledge, the only known {\it variation diminishing} bounds for maximal operators are the ones given in \eqref{Intro2.1} and \eqref{Intro3.1}, and the ones given in the recent work of Aldaz, Colzani and P\'{e}rez-L\'{a}zaro \cite{ACP} that show that the Lipschitz constant of a function (or H\"{o}lder constant) actually decreases under the action of the non-centered Hardy-Littlewood maximal function.

\medskip

Here we give a pool of affirmative answers to this question for two classical maximal operators of convolution type: the heat flow maximal operator (associated to the Gauss kernel) and the Poisson maximal operator (associated to the Poisson kernel). We shall obtain positive results for these two maximal operators in dimension $d=1$ for any $p\geq1$, and in dimensions $d>1$ for $p=2$ or $\infty$. We consider also the discrete versions of these operators and prove similar results. We start by reviewing the definitions and main properties of these convolution kernels, and as we move on the proofs, we shall see that the key idea to achieve these results is to explore the nice interplay between the maximal function analysis and the structure of the differential equations associated to these kernels (heat equation for the Gaussian and Laplace's equation for the Poisson kernel). 

\section{Main results}

\subsection{The continuous setting} We start by reviewing the definitions and stating the results in the context of the Euclidean space $\R^d$. All the results presented here hold for complex-valued functions, but for simplicity (since the maximal operators only see the absolute value of a function) we will work with real-valued functions.

\subsubsection{Heat flow maximal operator} Given $u_0 \in L^{p}(\R^d)$, $1\leq p \leq \infty$, we define $u:\R^d \times (0,\infty) \to \R$ by 
\begin{equation*}
u(x,t) = (u_0*K_t)(x)\,,
\end{equation*}
where $K_t$ is the heat kernel given by
\begin{equation}\label{heat}
K_t(x) = \frac{1}{(4\pi t)^{d/2}} e^{-|x|^2/4t}.
\end{equation}
In this case we know that $u \in C^{\infty}(\R^d \times (0,\infty))$ and solves the heat equation 
$$\partial_t u - \Delta u = 0 \ \ \textrm{in} \ \ \R^d \times (0,\infty),  $$
with
$$\lim_{t \to 0^+} u(x,t) = u_0(x) \ \ {\rm a.e.}\ x \in \R^d.$$
We consider here the maximal function associated to this heat flow, henceforth denoted by $*$ to facilitate the notation,
\begin{equation}\label{def_u*}
u^*(x)= \sup_{t >0} \,(|u_0| *K_t)(x).
\end{equation}
We shall prove the following regularity results for this  maximal operator.

\begin{theorem}\label{thm_cont} Let $u^*$ be the heat flow maximal function defined in \eqref{def_u*}. The following propositions hold.
\smallskip
\begin{itemize}
\item[(i)] Let $1 < p \leq \infty$ and $u_0 \in W^{1,p}(\R)$. Then $u^{*} \in W^{1,p}(\R)$ and 
\begin{equation*}
\|(u^*)'\|_p \leq \|u_0'\|_p.
\end{equation*}

\vspace{0.15cm}

\item[(ii)] Let $u_0 \in W^{1,1}(\R)$. Then $u^{*} \in L^{\infty}(\R)$ and has a weak derivative $(u^*)'$ that satisfies 
\begin{equation*}
\|(u^*)'\|_1 \leq \|u_0'\|_1.
\end{equation*}

\vspace{0.15cm}

\item[(iii)] Let $u_0$ be of bounded variation on $\R$. Then $u^*$ is of bounded variation on $\R$ and 
\begin{equation*}
V(u^*) \leq V(u_0).
\end{equation*}

\vspace{0.15cm}

\item[(iv)] Let $d>1$ and $u_0 \in W^{1,p}(\R^d)$, for $p =2$ or $p= \infty$. Then $u^{*} \in W^{1,p}(\R^d)$ and 
\begin{equation*}
\|\nabla u^*\|_p \leq \|\nabla u_0\|_p.
\end{equation*}

\vspace{0.25cm}

\end{itemize}
\end{theorem}

\subsubsection{Poisson maximal operator} Given $u_0 \in L^{p}(\R^d)$, $1\leq p \leq \infty$, we now define $u:\R^d \times (0,\infty) \to \R$ by 
\begin{equation*}
u(x,y) = (u_0*P_y)(x)\,,
\end{equation*}
where $P_y$ is the Poisson kernel for the upper half-space given by
\begin{equation}\label{Poisson}
P_y(x) = c_d  \,\frac{y}{(|x|^2 + y^2)^{(d+1)/2}},
\end{equation}
with 
$$c_d = \frac{\Gamma\left(\frac{d+1}{2}\right)}{\pi^{(d+1)/2}}.$$
In this case we know that $u \in C^{\infty}(\R^d \times (0,\infty))$ and that it solves Laplace's equation 
$$\Delta u = 0 \ \ \textrm{in} \ \ \R^d \times (0,\infty),  $$
where here we take the Laplacian with respect to the $(d+1)$ coordinates of $(x,y) = (x_1,x_2,...,x_d,y)$, as opposed to the notation used in the heat flow case, when we wrote the Laplacian only on the $x$-variable. Note also that
$$\lim_{y \to 0^+} u(x,y) = u_0(x) \ \ {\rm a.e.}\ x \in \R^d.$$
In other words, $u(x,y)$ is the harmonic extension of $u_0$ to the upper half-space. We consider here the maximal function associated to the Poisson kernel (henceforth denoted by the star $\star$, slightly different from the heat flow case)
\begin{equation}\label{def_u_star}
u^{\star}(x)= \sup_{y>0} \,(|u_0| *P_y)(x).
\end{equation}
With respect to this maximal operator we shall prove the following regularity results, in analogy with Theorem \ref{thm_cont}.

\begin{theorem}\label{thm_Poisson} Let $u^\star$ be the Poisson maximal function defined in \eqref{def_u_star}. The following propositions hold.
\smallskip
\begin{itemize}
\item[(i)] Let $1 < p \leq \infty$ and $u_0 \in W^{1,p}(\R)$. Then $u^{\star} \in W^{1,p}(\R)$ and 
\begin{equation*}
\|(u^\star)'\|_p \leq \|u_0'\|_p.
\end{equation*}

\vspace{0.15cm}

\item[(ii)] Let $u_0 \in W^{1,1}(\R)$. Then $u^{\star} \in L^{\infty}(\R)$ and has a weak derivative $(u^\star)'$ that satisfies 
\begin{equation*}
\|(u^\star)'\|_1 \leq \|u_0'\|_1.
\end{equation*}

\vspace{0.15cm}

\item[(iii)] Let $u_0$ be of bounded variation on $\R$. Then $u^\star$ is of bounded variation on $\R$ and 
\begin{equation*}
V(u^\star) \leq V(u_0).
\end{equation*}

\vspace{0.15cm}

\item[(iv)] Let $d>1$ and $u_0 \in W^{1,p}(\R^d)$, for $p =2$ or $p= \infty$. Then $u^{\star} \in W^{1,p}(\R^d)$ and 
\begin{equation*}
\|\nabla u^\star\|_p \leq \|\nabla u_0\|_p.
\end{equation*}

\vspace{0.15cm}

\end{itemize}
\end{theorem}

\subsection{The discrete setting} Again, all the results presented in this section hold for complex-valued functions, but for simplicity (since the discrete maximal operators only see the absolute value of a function) we will keep working with real-valued functions. For a bounded discrete function $f:\Z^d \to \R$ we define
$$\partial_{x_i} f(n) := f(n+e_i) - f(n),$$
where $n \in \Z^d$ and $e_i  = (0,. . .,1,. . . ,0)$ is the canonical $i$-th base vector with $1$ in the position $i$. The discrete gradient is then the vector
$$\nabla f(n)  = \left(\partial_{x_i} f(n), \partial f_{x_2}(n), . . . ,  \partial_{x_d} f(n)\right).$$
We let 
$$\|f\|_p = \left(\sum_{n \in \Z^d} |f(n)|^p\right)^{1/p},$$
if $1 \leq p < \infty$ and 
$$\|f\|_{\infty} = \sup_{n\in \Z^d} |f(n)|.$$
We define the $l^p(\Z^d)$-norms of the discrete function $n \mapsto |\nabla f(n)|$ in an analogous way (here $|\cdot|$ will always denote the usual Euclidean norm in $\R^d$). For a discrete function $f:\Z^d \to \R$ we define its {\it discrete Laplacian} $\Delta f:\Z^d \to \R$ as
$$\Delta f(n) = \frac{1}{2d} \sum_{\|m-n\|_{1}=1} \big[f(m)- f(n)\big],$$
where $\|n\|_1 = |n_1| + |n_2| + . . . + |n_d|$, if $n = (n_1, n_2,. . ., n_d) \in \Z^d$.

\subsubsection{Discrete heat flow maximal operator}  For a bounded function $u_0:\Z^d \to \R$, the {\it heat flow} in $\Z^d$ with initial condition $|u_0|$ is the unique (bounded and $C^1$ in time) solution $u: \Z^d \times [0,\infty) \to \R$ of  the differential equation
\begin{equation*}
\left\{ 
\begin{array}{rl}
\partial_t u (n,t)  &= \Delta u(n,t);\\
\\
u(n,0)&= |u_0(n)|.
\end{array}
\right.
\end{equation*}
This solution can be given in terms of a convolution with the discrete heat kernel (see for instance \cite{KN} and the references therein)
\begin{equation*}
u(n,t) = \big(|u_0|*\mc{K}_t \big)(n) = \sum_{m\in \Z^d} |u_0(n-m)| \,\mc{K}_t(m),
\end{equation*} 
where 
$$\mc{K}_t(m) = e^{-t}\,\prod_{i=1}^{d} I_{m_i}(t/d),$$
for $m = (m_1, m_2, . . . , m_d) \in \Z^d$, and where $I_k$, for an integer $k\geq 0$ and complex $z$, is the $I$-Bessel function defined by 
\begin{equation*}
I_k(z) = \sum_{j=0}^{\infty} \frac{(z/2)^{2j + k}}{j! \, (j +k)!} = \frac{1}{\pi} \int_0^{\pi} e^{z \cos \theta}\,\cos k\theta\, {\rm d} \theta,
\end{equation*}
and for a negative integer $k$ we put $I_k: = I_{-k}$. We note here that $I$ satisfies the differential equation
$$I_{k+1}(z) + I_{k-1}(z) = 2 I_{k}'(z),$$
which plainly implies that 
$$\partial_t \mc{K}_t(m)  = \Delta \mc{K}_t(m).$$
Moreover, for any $\theta \in \R$ and $z \in \C$, we also have \cite[Lemma 7]{KN}
$$\sum_{k = -\infty}^{\infty} \big[e^{-z}\,I_k(z)\big]\, e^{-i \theta k} = e^{z(\cos \theta - 1)},$$
which gives 
$$\|\mc{K}_t\|_{l^1(\Z^d)} = \sum_{m \in \Z^d} \mc{K}_t(m) = 1.$$
We then define the discrete heat flow maximal operator by 
\begin{equation}\label{def_u*_disc}
u^*(n) = \sup_{t >0} u(n,t) = \sup_{t>0} (|u_0|*\mc{K}_t \big)(n).
\end{equation}

Our next result shows that this maximal operator does not increase the $l^p$-variation in two situations: (i) in dimension $d=1$, for any $1\leq p\leq\infty$; (ii)  in dimension $d>1$, for $p=2$.

\begin{theorem}\label{thm_disc} Let $u^*$ be the discrete heat flow maximal function defined in \eqref{def_u*_disc}. The following propositions hold.

\vspace{0.1cm}

\begin{itemize}
\item[(i)] Let $1 \leq p \leq \infty$ and $u_0:\Z \to \R$ be a bounded discrete function such that $\|u_0'\|_p < \infty$. Then
$$\|(u^*)'\|_p \leq \|u_0'\|_p.$$

\vspace{0.15cm}

\item[(ii)] Let $u_0:\Z^d \to \R$ be a bounded discrete function such that $\|\nabla u_0\|_2 < \infty$. Then
$$\|\nabla u^*\|_2 \leq \|\nabla u_0\|_2.$$

\vspace{0.25cm}
\end{itemize}
\end{theorem}

\subsubsection{Discrete Poisson maximal operator} Recall that the continuous Poisson kernel $P_y(x)$ for the upper half-space defined in \eqref{Poisson} satisfies the semigroup property 
\begin{equation}\label{DP_sem}
P_{y_1} * P_{y_2} = P_{y_1 + y_2}
\end{equation}
for any $y_1, y_2>0$, and its Fourier transform verifies
$$\widehat{P_y}(\xi) = \int_{\R^n} P_y(x)\,e^{-2\pi i x \xi}\,\dx = e^{-2\pi y |\xi|} = \left(\widehat{P_1}(\xi) \right)^y.$$ 
Given a bounded discrete function $u_0:\Z^d \to \R$, we aim to lift this function to a harmonic function on the discrete upper half-space, i.e. we want to construct  $u: \Z^d \times \Z^{+} \to \R$ such that 
\begin{equation}\label{DP_eq-1}
\left\{ 
\begin{array}{rl}
\Delta u(n,y) = 0 &  {\rm in } \  \ \Z^d \times \N;\\
\\
u(n,0)&= u_0(n).
\end{array}
\right.
\end{equation}
Observe that here we use the $(d+1)$-dimensional discrete Laplacian, and that the parameter $y$ is now also discrete. We aim to accomplish this by convolving the initial datum $u_0$ with a certain integrable discrete kernel $\mc{P}_y$ (also denoted below by $\mc{P}(\cdot, y)$) that satisfies the semigroup property \eqref{DP_sem}. Let us proceed with a formal derivation of the kernel $\mc{P}_y$ first. 

\medskip

The function $\mc{P}: \Z^d \times \Z^{+} \to \R$ would have to satisfy
\begin{equation}\label{DP_eq0}
\left\{ 
\begin{array}{rl}
\Delta \mc{P}(n,y) = 0 &  {\rm in } \  \ \Z^d \times \N;\\
\\
\mc{P}(n,0)&= \delta_0(n),
\end{array}
\right.
\end{equation}
where $\delta_0$ is the function that is $1$, if $n=0$, and zero otherwise. Writing the harmonicity condition at the level $y=1$ we have
\begin{equation*}
2(d+1) \mc{P}_1(n) =  \mc{P}_0(n) + \mc{P}_2(n) + \sum_{\|m - n\|_1=1} \mc{P}_1(m).
\end{equation*}
We now multiply the last expression by $e^{-2\pi i  n \xi}$ and sum over $n \in \Z^d$ to get 
\begin{align*}
2(d+1) \widehat{\mc{P}_1}(\xi) & = 1 + \widehat{\mc{P}_2}(\xi) + \widehat{\mc{P}_1}(\xi) \sum_{k=1}^d \big(e^{2\pi i \xi_k} + e^{-2\pi i \xi_k}\big),
\end{align*}
where $\xi  = (\xi_1, \xi_2,..., \xi_d)$ belongs to the torus $\mathbb{T}^d = [-\tfrac12,\tfrac12]^d$. From the semigroup property we have $\widehat{\mc{P}_2}(\xi) = \widehat{\mc{P}_1}(\xi)^2$ and thus
\begin{equation}\label{DP_eq1}
 \widehat{\mc{P}_1}(\xi)^2 - \widehat{\mc{P}_1}(\xi)\left( 2(d+1) - 2\sum_{k=1}^d \cos(2 \pi \xi_k)\right)+ 1 = 0.
\end{equation}
From \eqref{DP_eq1} we conclude that 
\begin{equation}\label{DP_choice}
\widehat{\mc{P}_1}(\xi) = \left( (d+1) - \sum_{k=1}^d \cos(2 \pi \xi_k)\right) \pm \left(\left((d+1) - \sum_{k=1}^d \cos(2 \pi \xi_k)\right)^2 - 1 \right)^{1/2},
\end{equation}
where the choice of signs could (at least in principle) be taken in any measurable way. 

\medskip

We shall pick the negative sign in \eqref{DP_choice} and {\it define} our discrete Posson kernel by the following three expressions:
\begin{equation}\label{DP_eq2}
\widehat{\mc{P}_1}(\xi) := \left( (d+1) - \sum_{k=1}^d \cos(2 \pi \xi_k)\right) - \left(\left((d+1) - \sum_{k=1}^d \cos(2 \pi \xi_k)\right)^2 - 1 \right)^{1/2},
\end{equation}
\begin{equation}\label{DP_eq3}
\widehat{\mc{P}_y}(\xi) := \left(\widehat{\mc{P}_1}(\xi) \right)^y,
\end{equation}
for $y \geq 0$, and
\begin{equation}\label{DP_eq4}
\mc{P}_y(n) = \int_{\mathbb{T}^d} \widehat{\mc{P}_y}(\xi) \,e^{2 \pi i n \xi}\,\dxi.
\end{equation}
Observe from \eqref{DP_eq2} that $0 < \widehat{\mc{P}_1}(\xi)\leq 1$ for all $\xi \in \mathbb{T}^d$, with $\widehat{\mc{P}_1}(\xi) = 1$ if and only if $\xi=0$. Thus, for  $\xi \in \mathbb{T}^d \setminus \{0\}$ the function $y \mapsto \widehat{\mc{P}_y}(\xi)$ is decreasing (and goes to $0$ as $y \to \infty$). From \eqref{DP_eq2}, \eqref{DP_eq3} and \eqref{DP_eq4} it is clear that $\mc{P}: \Z^d \times \Z^{+} \to \R$ satisfies \eqref{DP_eq0}, and therefore the maximum principle holds.

\medskip

From \eqref{DP_eq4} we have $|P_y(n)| \leq 1$ for all $(n,y) \in \Z^d \times \Z^{+}$. We want to show now that we also have $\mc{P}_y(n) \geq 0$ for all $(n,y) \in \Z^d \times \Z^{+}$. For this we use the maximum principle. Given $\varepsilon >0$ observe first that, by Lebesgue dominated convergence, there is a $y_0 >0$ such that 
\begin{equation*}
\big|\mc{P}_{y_0}(n)\big| \leq \int_{\mathbb{T}^d}  \widehat{\mc{P}_{y_0}}(\xi) \,\dxi < \varepsilon
\end{equation*}
for any $n \in \Z^d$. Now, by the Riemann-Lebesgue lemma, there exists a radius $r_0 >0$ such that 
\begin{equation*}
\big|\mc{P}_{y}(n)\big| < \varepsilon
\end{equation*}
for all $0 \leq y \leq y_0$ and $|n| \geq r_0$. When we consider the cylindrical contour of radius $r_0$, delimited by the hyperplanes $y=0$ and $y=y_0$, by the maximum principle we have
\begin{equation}\label{DP_eq5}
\mc{P}_{y}(n) > -\varepsilon
\end{equation}
within this region. Since we could take $y_0$ as large as we wanted (and then $r_0$ large as well) we conclude that \eqref{DP_eq5} actually holds for all $(n,y) \in \Z^d \times \Z^{+}$. Since $\varepsilon>0$ is arbitrary we have
\begin{equation}\label{DP_eq6}
\mc{P}_{y}(n) \geq 0
\end{equation}
for all $(n,y) \in \Z^d \times \Z^{+}$ as desired. 

\medskip

Since $\widehat{\mc{P}_{y}}(0)=1$ for all $y \geq 0$, it follows from \eqref{DP_eq6} (and a standard approximation argument using the smoothing F\'{e}jer kernel) that $\mc{P}_y$ is integrable and 
\begin{equation*}
\sum_{n\in \Z^d} \mc{P}_y(n) = 1
\end{equation*}
for all $y \geq 0$. Therefore, given a bounded discrete function $u_0:\Z^d \to \R$ the function 
$$u(n,y) = u_0 * \mc{P}_y(n) = \sum_{m \in \Z^d} u_0(n-m)\mc{P}_y(m)$$
is well defined, and one now clearly sees that it satisfies \eqref{DP_eq-1}.

\medskip

We now define our discrete Poisson maximal operator (keeping the notation $\star$) by 
\begin{equation}\label{DP_eq7}
u^{\star}(n) = \sup_{y \in \Z^{+}} (|u_0| * \mc{P}_y)(n).
\end{equation}
By a classical result of Stein \cite[Theorem 1]{S1} the maximal operator $\star$ verifies the $l^p$-boundedness
\begin{equation}\label{Stein_lp_bound}
\|u^{\star}\|_{l^p(\Z^d)} \leq C_p \|u\|_{l^p(\Z^d)}
\end{equation}
for $1 < p \leq \infty$. With respect to this maximal function we prove the following regularity properties, in analogy with Theorem \ref{thm_disc}.

\begin{theorem}\label{thm_disc_Poisson} Let $u^\star$ be the discrete Poisson maximal function defined in \eqref{DP_eq7}. The following propositions hold.

\vspace{0.1cm}

\begin{itemize}
\item[(i)] Let $1 \leq p \leq \infty$ and $u_0:\Z \to \R$ be a bounded discrete function such that $\|u_0'\|_p < \infty$. Then
$$\|(u^\star)'\|_p \leq \|u_0'\|_p.$$

\vspace{0.15cm}

\item[(ii)] Let $u_0:\Z^d \to \R$ be a bounded discrete function such that $\|\nabla u_0\|_2 < \infty$. Then
$$\|\nabla u^{\star}\|_2 \leq \|\nabla u_0\|_2.$$

\vspace{0.25cm}
\end{itemize}
\end{theorem}

We now move on to the proofs of these results. We shall start with the discrete cases, that are technically simpler but already give a good flavor of the main ideas that shall be used in the continuous cases. The key insight here is that all of these maximal functions have the property of being subharmonic in the set where they disconnect from the original function. This shall be obtained by exploiting the structure of the underlying partial differential equations.

%%%%%%%%%%%%%%%%%%%%%%%%%%%%%%%%%%%%%%%%%%%%%%%%%%%%%%%%%%%%%%%%%%%%%%%%%%%%%%%%%%%%%%%%%%%%%%%%%%%%%%%%%%%%%%%%%%%%%%%%%%%%%%%%%%%%%%%%%%%%%%%%%%%%%%%%

\section{Proof of Theorem \ref{thm_disc} - Discrete heat kernel}

From now on we assume, without loss of generality, that $u_0 \geq0$, since $\|\nabla |u_0|\|_p \leq \|\nabla u_0\|_p$ for any $1\leq p \leq \infty$.

\subsection{Preliminaries}  The essence of the following lemma is the fact that $u^*$ is subharmonic in the set where it disconnects from $u_0$. The statement in the following format will be more convenient later in the proof.

\begin{lemma}\label{lem3}
Let $u_0 \in l^{\infty}(\Z^d)$ and $u^*$ be its discrete heat flow maximal function. Let $f \in l^{\infty}(\Z^d)$ be such that $f \geq u_0$. Let $I = \{n \in \Z^d;\, u^*(n)\leq f(n)\}$ and suppose that $\Delta f(n) \leq 0$ for all $n \in I^c$. Then $I = \Z^d$.
\end{lemma}

\begin{proof}
Suppose without loss of generality $u_0$ is not identically zero. Let $u: \Z^d \times [0,\infty)$ be the heat flow in $\Z^d$ with initial condition $u_0$. It suffices to prove that 
$$\sup_{t \geq 0} u(n,t) \leq f(n),$$
for all $n \in I^c$. It is easy to see that $\|u(\cdot,t)\|_{\infty} \leq \|u_0\|_{\infty}$ for any $t\geq0$. Therefore,
\begin{equation}\label{lem_Sec2.1}
|\partial_t u(n,t)|\leq \|\Delta u\|_{\infty}\leq 2 \|u_0\|_{\infty}. 
\end{equation}
Take $\varepsilon >0$ and define 
$$T_{\varepsilon} = \sup\{t \geq 0;\ u(n,t) \leq f(n) + \varepsilon,\ \forall n \in I^c\}.$$
In view of \eqref{lem_Sec2.1} we have 
\begin{equation*}
T_{\varepsilon} \geq \frac{\varepsilon}{2\|u_0\|_{\infty}}> 0.
\end{equation*}
Suppose $T_{\varepsilon} < \infty$. For any $n \in I^c$ and $0 \leq t \leq T_{\varepsilon}$ we have
\begin{align*}
\partial_t u(n,t)&=\frac{1}{2d}\sum_{\|m-n\|_1=1} \big[u(m,t)- u(n,t)\big]\\
    &\leq\frac{1}{2d}\sum_{\|m-n\|_1=1}\big[f(m)+\varepsilon-u(n,t)\big]\\
    &=f(n)+\varepsilon-u(n,t)+\frac{1}{2d}\sum_{\|m-n\|_1=1}\big[f(m)-f(n)\big]\\
    &\leq f(n)+\varepsilon-u(n,t).
\end{align*}
Define $y(t)=u(n,t)-(f(n)+\varepsilon)$. From the inequality above we have $y'(t)\leq-y(t)$ and thus, for any $t\in [0,T_\varepsilon]$,
$$ y(t)\leq e^{-t}y(0)\leq -\varepsilon e^{-t}\leq-\varepsilon e^{-T_\varepsilon}.$$
Hence
\begin{equation}\label{lem_Sec2.1.2}
u(n,T_\varepsilon)\leq (f(n)+\varepsilon)-\frac{\varepsilon}{\exp(T_\varepsilon)}
\end{equation}
for any $n\in I^c$. Combining \eqref{lem_Sec2.1.2} with \eqref{lem_Sec2.1} we conclude that
$u(n,t)\leq f(n) + \varepsilon$ for any $t$ with
$$T_\varepsilon\leq t\leq T_\varepsilon+\frac{\varepsilon}{2\|u_0\|_\infty
   \exp(T_\varepsilon)},$$
which is in contradiction with the assumption that $T_\varepsilon < \infty$. Hence $T_\varepsilon=\infty$ and 
$$ u(n,t)\leq f(n)+\varepsilon$$
for all $t\geq 0$ and all $n \in I^c$. To conclude the proof, take the limit $\varepsilon\to0^{+}$.
\end{proof}

\subsection{Proof of Theorem \ref{thm_disc}}

\subsubsection*{Step 1: Zorn's lemma} Recall that we are working here in the two cases: (i) $d=1$ and $1\leq p \leq \infty$ or (ii) $d>1$ and $p=2$. Consider the following family of functions:
\begin{equation*}
\mc{S} = 
\left\{
\begin{array}{cc}
f:\Z^d \to \R;\\
\\
u_0(n) \leq f(n) \leq u^*(n);& \ \ \forall n \in \Z^d;\\
\\
\|\nabla f\|_p \leq \|\nabla u_0\|_p\,.&
\end{array}
\right.
\end{equation*}
Note that $\mc{S}$ is non-empty since $u_0 \in \mc{S}$. We want to show ultimately that $u^* \in \mc{S}$. We put a partial order $\preceq$ in $\mc{S}$ by considering the pointwise order (i.e. $f \preceq g$ in $\mc{S}$ if and only if $f(n) \leq g(n)$ for all $n \in \Z^d$). Let us prove that $(\mc{S,\preceq)}$ is inductive, i.e. every totally ordered subset has an upper bound in $\mc{S}$. Let $\{f_{\alpha}\}_{\alpha \in \Lambda}$ be a totally ordered subset and define
$$\overline{f}(n) = \sup_{\alpha \in \Lambda}f_{\alpha}(n).$$
We claim that $\overline{f} \in \mc{S}$. It is clear that $u_0(n) \leq \overline{f}(n) \leq u^*(n)$ for all $n \in \Z^d$. Let $J \subset \Z^d$ be a {\it finite} set and define $\widetilde{J} = \{n \in \Z^d;\ {\rm dist}(n,J) \leq 1\}$ (this distance is taken with respect to the $\|\cdot\|_1$-norm). There exists a sequence $\{f_k\}_{k \in \N} \subset \{f_{\alpha}\}_{\alpha \in \Lambda}$ such that 
$$\lim_{k \to \infty} f_k(n) = \overline{f}(n)$$
for all $n \in \widetilde{J}$. Thus,
\begin{equation*}
\sum_{n \in J} \big|\nabla \overline{f}(n)\big|^p = \lim_{k \to \infty} \sum_{n \in J} |\nabla f_k(n)|^p \leq \limsup_{k \to \infty} \|\nabla f_k\|_p^p \leq \|\nabla u_0\|_p^p.
\end{equation*}
Since this holds for any finite set $J$, we must have 
$$\big\|\nabla \overline{f}\big\|_p \leq \|\nabla u_0\|_p,$$
and thus $\overline{f} \in \mc{S}$. By Zorn's lemma $(\mc{S}, \preceq)$ has (at least) one maximal element, which we call $g$.

\subsubsection*{Step 2: Conclusion} We now claim that $g = u^*$. Suppose this is not true and let $I = \{n \in \Z^d;\ g(n) = u^*(n)\}$. By Lemma \ref{lem3}, we know that $g$ cannot be superharmonic on $I^c$, and thus there is a point $n \in I^c$ such that $\Delta g(n) >0$.

\medskip

We first deal with the case $d \geq 1$ and $p=2$. Consider the function
$$q_n(x) = \sum_{\|m-n\|_1=1} (g(m) - x)^2.$$
It is easy to see that $x \mapsto q_n(x)$ is a strictly convex function with its unique minimizer $x = x_n$ given by 
$$x_n = \frac{1}{2d} \sum_{\|m-n\|_1=1} g(m) = g(n) + \Delta g (n).$$
Therefore, if $\Delta g(n) >0$, we can consider the function
\begin{equation*}
\widetilde{g}(m) = \left\{
\begin{array}{lr}
g(m),& \ \ {\rm if} \ \ m \neq n;\\
\\
\min\big\{u^*(n), \,g(n) + \Delta g(n)\big\},& \ \ {\rm if}\ m =n.
\end{array}
\right.
\end{equation*}
Then $g \leq \widetilde{g} \leq u^*$ pointwise and, since $g(n) < \widetilde{g}(n) \leq x_n$, the strict convexity of $q_n$ gives us
$$q_n(g(n)) > q_n\big(\widetilde{g}(n)\big) \geq q_n(x_n).$$
This plainly implies that
$$\|\nabla g\|_2^2 = \big\|\nabla \widetilde{g}\big\|_2^2 + q_n(g(n)) - q_n\big(\widetilde{g}(n)\big) > \big\|\nabla \widetilde{g}\big\|_2^2.$$
Thus $\widetilde{g} \in \mc{S}$ and this contradicts the maximality of $g$.

\medskip

Now we deal with the case $d=1$ and $1\leq p \leq \infty$. The idea is the same as above. If $1\leq p < \infty$ we simply observe that the function (for fixed $a,b \in \R$)
$$q(x) = |a-x|^p + |b-x|^p$$
is convex (strictly convex if $1<p$) with minimizer $x = (a+b)/2$. If $p=\infty$ we note that 
$$q(x) = \max\{|a-x|, |b-x|\}$$
is also convex with minimizer $x = (a+b)/2$. This concludes the proof.

%%%%%%%%%%%%%%%%%%%%%%%%%%%%%%%%%%%%%%%%%%%%%%%%%%%%%%%%%%%%%%%%%%%%%%%%%%%%%%%%%%%%%%%%%%%%%%%%%%%%%%%%%%%%%%%%%%%%%%%%%%%%%%%%%

\section{Proof of Theorem \ref{thm_disc_Poisson} - Discrete Poisson kernel}
In this section we keep, without loss of generality, the assumption $u_0 \geq0$ (since $\|\nabla |u_0|\|_p \leq \|\nabla u_0\|_p$ for any $1\leq p \leq \infty$).

\subsection{Preliminaries} We start by proving the analogous statement to Lemma \ref{lem3} for the discrete Poisson maximal function.

\begin{lemma}\label{lem7_Poisson_b}
Let $u_0 \in l^{\infty}(\Z^d)$ and $u^{\star}$ be its discrete Poisson maximal function. Let $f \in l^{\infty}(\Z^d)$ be such that $f \geq u_0$. Let $I = \{n \in \Z^d;\, u^{\star}(n)\leq f(n)\}$ and suppose that $\Delta f(n) \leq 0$ for all $n \in I^c$. Then $I = \Z^d$.
\end{lemma}

\begin{proof}
For $y\in\Z^+$ we define
$$ \beta (y)=\sup_{n\in\Z^d} (u(n,y)-f(n))_+ \,,$$
where $t_+ := \max\{t,0\}$. Observe that $\beta(0)=0$ and that, for any $y$, we have
$$0\leq\beta(y)\leq\|u^{\star}\|_\infty+\|f\|_\infty <\infty$$ 
and
\begin{equation}\label{pf_lem6_new_eq1}
u(\cdot,y)\leq f+\beta(y).
\end{equation}
 
\medskip

Suppose that $I^c\neq\emptyset$. In this case,
$$ \beta(y)=\sup_{n\in I^c} \;( u(n,y)-f(n))_+.$$
Take $n\in I^c$ and $y>0$. Using the fact that $u$ is harmonic in the discrete upper half-space, together with \eqref{pf_lem6_new_eq1} and the hypothesis that $\Delta f(n) \leq 0$ for $n \in I^c$, we find
\begin{align*}
    2(d+1)u(n,y)&=u(n,y-1)+u(n,y+1)+\sum_{||m-n||_1=1}u(m,y)\\
    &\leq 2 f(n)+\beta(y-1) + \beta(y+1)+\sum_{||m-n||_1=1}\big(f(m)+\beta(y)\big)\\
   &\leq 2(d+1)f(n)+\beta(y-1)+\beta(y+1)+2d \,\beta(y).
  \end{align*}
Therefore
\begin{equation}\label{pf_lem6_new_eq2}
2(d+1)\big(u(n,y)-f(n)\big)\leq\beta(y-1)+\beta(y+1)+2d\, \beta(y). 
\end{equation}
Taking the supremum over $n\in I^c$ on the left hand side of \eqref{pf_lem6_new_eq2} we find
 $$ \beta(y)\leq \dfrac{\beta(y-1)+\beta(y+1)}{2}$$
 for any $y >0$. Since the sequence $y \mapsto \beta(y)$ is bounded and $\beta(0)=0$, we must have $\beta(y)=0$ for all $y>0$. Therefore $I^c$ would be empty, in contradiction with the original assumption $I^c\neq\emptyset$. 
\end{proof}

\subsection{Proof of Theorem \ref{thm_disc_Poisson}} 
Once we have established Lemma \ref{lem7_Poisson_b}, the proof of Theorem \ref{thm_disc_Poisson} plainly follows by the argument based on Zorn's lemma used in the proof of Theorem \ref{thm_disc} for the discrete heat flow maximal function. We will omit the details.

\section{Proof of Theorem \ref{thm_cont} - Continuous heat kernel}

\subsection{Preliminaries} We begin this section with a selection of lemmas that will be helpful as we move on to the proof. Throughout this section we assume without loss of generality that $u_0 \geq0$, for if $u_0 \in W^{1,p}(\R^d)$ we have $|u_0| \in W^{1,p}(\R^d)$ and $|\nabla |u_0|| = |\nabla u_0|$ a.e. if $u_0$ is real-valued (in the general case $u_0$ complex-valued we have $|\nabla |u_0|| \leq |\nabla u_0|$ a.e), and if $u_0$ is of bounded variation on $\R$ we have $V(|u_0|) \leq V(u_0)$. In what follows we write
$${\rm Lip}(u) = \sup_{\stackrel{x,y \in \R^d}{x\neq y}} \frac{|u(x) - u(y)|}{|x-y|}$$
for the Lipschitz constant of a function $u:\R^d \to \R$.

\begin{lemma}\label{lem4}.
\begin{itemize}
\item[(i)] If $u_0 \in C(\R^d) \cap L^p(\R^d)$, for some $1\leq p < \infty$, then $u^* \in C(\R^d)$. 
\smallskip
\item[(ii)] If $u_0$ is bounded and Lipschitz continuous then $u^*$ is bounded and Lipschitz continuous with ${\rm Lip}(u^*) \leq {\rm Lip}(u_0)$.
\end{itemize}
\end{lemma}
\begin{proof} 
(i) Recall that 
$$u^*(x) = \sup_{t >0} (u_0*K_t)(x)$$
with the heat kernel $K_t$ defined in \eqref{heat}. Let us denote here $\tau_hu_0(x):= u_0(x- h)$. Given $\varepsilon >0$, there is a time $t_{\varepsilon} < \infty$ such that 
\begin{align*}
|\tau_h u_0 - u_0|*K_t(x)& \leq \left(|\tau_h u_0 - u_0|^p*K_t(x)\right)^{1/p}\\
& \leq \left(\|\tau_h u_0 - u_0\|_p^p \ \|K_t\|_{\infty}\right)^{1/p} = \frac{\|\tau_h u_0 - u_0\|_p}{(4\pi t)^{d/2p}}\\
& \leq \frac{2\|u_0\|_p}{(4\pi t)^{d/2p}} < \varepsilon
\end{align*}
whenever $t > t_{\varepsilon}$, for all $x, h \in \R^d$. Note that we used Jensen's inequality in the first line above and Young's inequality on the second line. On the other hand, given $x \in \R^d$ and if $0 < t \leq t_{\varepsilon}$, we can choose $\delta >0$ such that
\begin{align*}
& |\tau_h u_0 - u_0|*K_t(x) \\
&= \int_{|y| < \sqrt{ t_{\varepsilon}}} |\tau_h u_0 - u_0|(x-y)\,K_t(y)\,\dy + \int_{|y| \geq \sqrt{t_{\varepsilon}}} |\tau_h u_0 - u_0|(x-y)\,K_t(y)\,\dy\\
& \leq \sup_{w \in B_{\sqrt{t_{\varepsilon}}}(x)}|\tau_h u_0 - u_0|(w) + \|\tau_h u_0 - u_0\|_p\,\|\chi_{\{|y|\geq \sqrt{t_{\varepsilon}}\}} \,K_t\|_{p'}< \varepsilon
\end{align*}
whenever $|h| < \delta$, where we have used the fact that $\|\chi_{\{|y|\geq \sqrt{t_{\varepsilon}}\}}\, K_t\|_{p'}$ is bounded for $0 < t \leq t_{\varepsilon}$. Using the sublinearity, we then arrive at 
$$\big|\tau_h u^*(x) - u^*(x)\big| \leq (\tau_h u_0 - u_0)^*(x) \leq \varepsilon$$
for $|h| < \delta$, which proves that $u^*$ is continuous at $x$.

\medskip

\noindent (ii) It is easy to check that if $u_0$ is bounded by $M$ and has Lipschitz constant $L$, then for each time $t>0$ the function $u_0 * K_t$ is also bounded by $M$ and admits the same Lipschitz constant $L$. In this case, the pointwise supremum of uniformly Lipschitz functions is still Lipschitz with (at most) the same constant.
\end{proof}

We will say here that a continuous function $f$ is {\it subharmonic} in an open set $A$ if, for every $x \in A$, and every ball $\overline{B_r(x)} \subset A$ we have
$$f(x) \leq \frac{1}{\sigma_{d-1}}\int_{S^{d-1}} f(x +r\xi)\,{\rm d}\sigma(\xi),$$
where $\sigma_{d-1}$ denotes the surface area of the unit sphere $S^{d-1}$, and ${\rm d}\sigma$ is its surface measure. Here $B_r(x)$ denotes the open ball of radius $r$ and center $x$, and $\overline{B_r(x)}$ denotes the corresponding closed ball.

\begin{lemma}[Subharmonicity]\label{lem5} Let $u_0 \in C(\R^d) \cap L^p(\R^d)$ for some $1\leq p < \infty$ or $u_0$ be bounded and Lipschitz continuous. Then $u^*$ is subharmonic in the open set $A = \{x \in \R^d; \,u^*(x) > u_0(x)\}$.
\end{lemma}

\begin{proof}
Note that by Lemma \ref{lem4} we have $u^*$ continuous and thus the set $A = \{x \in \R^d; \,u^*(x) > u_0(x)\}$ is in fact open. Take $x_0 \in A$ and a radius $r>0$ such that $\overline{B_r(x_0)} \subset A$. Let $h:\overline{B_r(x_0)} \to \R$ be the solution of the Dirichlet boundary value problem
\begin{equation*}
\left\{
\begin{array}{rl}
\Delta h = 0& \ {\rm in} \ B_r(x_0);\\
h = u^*& \ {\rm in} \ \partial B_r(x_0).
\end{array}
\right.
\end{equation*}
Since $u^*$ is a continuous function, this problem does admit a unique solution $h \in C^{2}(B_r(x_0)) \cap C\big(\overline{B_r(x_0)}\big)$. Now let $T>0$ and $\Omega = B_r(x_0) \times(0,T)$. Observe that $v(x,t) := u(x,t) - h(x) \in C^2(\Omega)\cap C\big(\overline{\Omega}\big)$ and solves the heat equation in $\Omega$. By the maximum principle for the heat equation, the maximum of $v$ in $\overline{\Omega}$ must be attained either in $\partial B_r(x_0) \times [0,T]$ or in $\overline{B_r(x_0)} \times \{t=0\}$. By construction, note that 
$$\max_{\partial B_r(x_0) \times [0,T]} v(x,t) =  \max_{\partial B_r(x_0) \times [0,T]} u(x,t) - u^*(x) \leq 0.$$
Let $y_0$ be such that 
$$\max_{\overline{B_r(x_0)}} v(x,0) = v(y_0,0).$$
We claim that $v(y_0,0) \leq 0$. In fact, let us suppose that $v(y_0,0) > 0$. Then, by the maximum principle, $v(y_0,t) \leq v(y_0,0)$ for any $0\leq t \leq T$, which in turn implies that $u(y_0,t) \leq u_0(y_0)$ for any $0\leq t \leq T$. Since $T$ is arbitrary, we would have $u^*(y_0) = u_0(y_0)$ and thus $y_0 \notin A$, contradiction. Therefore
$$\max_{\overline{B_r(x_0)}} v(x,0)\leq 0,$$
which plainly gives
$$u(x_0, t) \leq h(x_0)$$
for any $0\leq t \leq T$. As $T$ is arbitrary we conclude that 
$$u^*(x_0) \leq h(x_0),$$
which is the desired result since $h$ is harmonic and thus equal to its average over the sphere $\partial B_r(x_0)$, where $h = u^*$ by construction.
\end{proof}

The next lemma will be important in the proof of the theorem for $p=2$ and $d\geq 1$.

\begin{lemma}\label{lem6}
Let $f, g \in C(\R^d) \cap W^{1,2}(\R^d)$ with $g$ Lipschitz. Suppose that $g \geq 0$ and that $f$ is subharmonic in the open set $J = \{x\in \R^d; \ g(x) >0\}$. Then
\begin{equation*}
\int_{\R^d} \nabla f\,.\,\nabla g\,\,\dx\leq 0.
\end{equation*}
\end{lemma}
\begin{proof}
Formally, the identity 
$$\int_{\R^d} \nabla f\,.\,\nabla g\,\,\dx = \int_{\R^d} (-\Delta f)\,g\,\dx$$
would imply the result, since $f$ subharmonic in the set $\{g>0\}$ would mean $-\Delta f \leq 0$ in this set. This justifies the intuition for the result. Our work here is to make this argument rigorous. 

\medskip

Our first claim is that we can assume without loss of generality that $g$ has compact support. To see this let $\Psi \in C^{\infty}_c(\R^d)$ be a non-negative function such that $\Psi(x) \equiv 1$ in $B_1$ and $\supp(\Psi) \subset B_2$. For $N \in \N$, let $\Psi_N(x) := \Psi(x/N)$ and consider $g_N(x) := g(x)\,\Psi_N(x)$. If we could prove the result for each $g_N$ (note that $f$ is subharmonic in the set $\{g_N>0\}$ and each $g_N$ is Lipschitz since $g$ is bounded), since $g_N \to g$ in $W^{1,2}(\R^d)$ we would obtain 
\begin{equation*}
\int_{\R^d} \nabla f\,.\,\nabla g\,\,\dx = \lim_{N\to \infty} \int_{\R^d} \nabla f\,.\,\nabla g_N\,\,\dx \leq 0. 
\end{equation*}
From now on we assume that $\supp(g) \subset B_R$ for some $R>0$.

\medskip

Consider a non-negative function $\varphi \in C^{\infty}_c(\R^d)$ with support on the unit ball $B_1$ and integral $1$. For $\varepsilon >0$ put $\varphi_{\varepsilon}(x) = \varepsilon^{-d}\varphi(x/\varepsilon)$ and write
$$f_{\varepsilon} = f *\varphi_{\varepsilon}.$$
We see that $f_{\varepsilon} \in C^{\infty}(\R^d)$ and it is not hard to check that 
\begin{equation*}
\partial_{x_i} f_{\varepsilon} =  (\partial_{x_i} f) *\varphi_{\varepsilon} = f*(\partial_{x_i}\varphi_{\varepsilon}),
\end{equation*}
and
\begin{equation}\label{lem5_laplace}
\partial_{x_ix_i} f_{\varepsilon} =  (\partial_{x_i} f) *(\partial_{x_i} \varphi_{\varepsilon}) = f*(\partial_{x_ix_i}\varphi_{\varepsilon}).
\end{equation}

\smallskip

\noindent Let us define the set $J_{\varepsilon} = \{x \in J;\ {\rm dist}(x,\partial J) >\varepsilon\}$. A simple computation shows that $f_{\varepsilon}$ is subharmonic on $J_{\varepsilon}$. In fact, if $x \in J_{\varepsilon}$ and $\overline{B_r(x)}\subset J_{\varepsilon}$, we have
\begin{align*}
f_{\varepsilon}(x) &= \int_{B_{\varepsilon}} f(x - y) \, \varphi_{\varepsilon}(y)\,\dy\\
& \leq  \int_{B_{\varepsilon}} \frac{1}{\sigma_{d-1}}\int_{S^{d-1}} f(x -y +r\xi)\,{\rm d}\sigma(\xi) \,\varphi_{\varepsilon}(y)\,\dy\\
& =  \frac{1}{\sigma_{d-1}}\int_{S^{d-1}} f_{\varepsilon} (x +r\xi)\,{\rm d}\sigma(\xi).
\end{align*}
Since $f_{\varepsilon} \in C^{\infty}(\R^d)$, this implies that $(-\Delta f_{\varepsilon})(x) \leq 0$ for $x \in J_{\varepsilon}$.

\medskip

For $\varepsilon >0$ and $\psi \in C^{\infty}_c(\R^d)$ we can apply integration by parts to get
\begin{equation*}
\int_{\R^d} \nabla f_{\varepsilon}\,.\,\nabla \psi\,\,\dx = \int_{\R^d} (-\Delta f_{\varepsilon})\,\psi\,\dx.
\end{equation*}
Now since $|\nabla f_{\varepsilon}| \in L^2(\R^d)$ and $\Delta f_{\varepsilon} \in L^2(\R^d)$, we might approximate our function $g \in W^{1,2}(\R^d)$ by such $\psi \in C^{\infty}_c(\R^d)$  (in the $W^{1,2}$-norm) to obtain
\begin{align}\label{lem5_limit}
\begin{split}
\int_{\R^d} \nabla f_{\varepsilon}\,.\,\nabla g\,\,\dx & = \int_{\R^d} (-\Delta f_{\varepsilon})\,g\,\dx\\
& = \int_{J \setminus J_{\varepsilon}} (-\Delta f_{\varepsilon})\,g\,\dx + \int_{J_{\varepsilon}} (-\Delta f_{\varepsilon})\,g\,\dx \\
& \leq \int_{J\setminus J_{\varepsilon}} (-\Delta f_{\varepsilon})\,g\,\dx.
\end{split}
\end{align}
Let $C$ be the Lipschitz constant of $g$. Then, for any $x \in J\setminus J_{\varepsilon}$ we have 
\begin{equation}\label{lem5_eq2}
|g(x)| \leq C \varepsilon.
\end{equation}
From \eqref{lem5_laplace} we observe that 
\begin{equation}\label{lem5_eq3}
\partial_{x_ix_i} f_{\varepsilon} = (\partial_{x_i} f) *\left(\varepsilon^{-1} (\partial_{x_i} \varphi)_{\varepsilon}\right),
\end{equation}
where $(\partial_{x_i} \varphi)_{\varepsilon}(x) = \varepsilon^{-d} (\partial_{x_i} \varphi)(x/\varepsilon)$. From \eqref{lem5_eq2} and \eqref{lem5_eq3}, using H\"{o}lder's inequality and Young's inequality we get 
\begin{align}\label{lem5_limit2}
\begin{split}
\int_{J \setminus J_{\varepsilon}} &\big|(-\Delta f_{\varepsilon})\,g\big|\,\dx  \leq C \int_{J \setminus J_{\varepsilon}} \sum_{i=1}^d  \big|(\partial_{x_i} f) *(\partial_{x_i}\varphi)_{\varepsilon} \big|\,\dx\\
& \leq C\,d \int_{J \setminus J_{\varepsilon}}  |\nabla f|*(|\nabla\varphi|)_{\varepsilon} \,\,\dx\\
& \leq C\,d \left(\int_{J \setminus J_{\varepsilon}} \big| |\nabla f|*(|\nabla\varphi|)_{\varepsilon}\big|^2 \,\,\dx\right)^{1/2} \, m(J \setminus J_{\varepsilon})^{1/2}\\
& \leq C\,d \ \|\nabla f\|_2 \ \|\nabla\varphi\|_1 \ m(J \setminus J_{\varepsilon})^{1/2}.
\end{split}
\end{align}
Since $\supp(g) \subset B_R$, we know that $m(J \setminus J_{\varepsilon}) \to 0$ as $\epsilon \to 0$. Finally, since $f_{\varepsilon} \to f$ in $W^{1,2}(\R^d)$ we use \eqref{lem5_limit} and \eqref{lem5_limit2} to get
\begin{align*}
\int_{\R^d} \nabla f\,.\,\nabla g\,\,\dx & = \lim_{\varepsilon \to 0} \int_{\R^d} \nabla f_{\varepsilon}\,.\,\nabla g\,\,\dx\\
& \leq \lim_{\varepsilon \to 0} \int_{J \setminus J_{\varepsilon}} (-\Delta f_{\varepsilon})\,g\,\dx = 0.
\end{align*}
\end{proof}

\begin{lemma}[Reduction to the Lipschitz case]\label{lem7}
In order to establish Theorem \ref{thm_cont} - parts (i) and (iv) - it suffices to consider the initial datum $u_0$ Lipschitz.
\end{lemma}
\begin{proof}
If $p=\infty$, recall that a function $u_0 \in W^{1,\infty}(\R^d)$ can be modified on a set of measure zero to become bounded and Lipschitz continuous. 

\medskip

If $1 < p <\infty$, we take $\varepsilon >0$ and consider
$$u_{\varepsilon}(x) = u_0*K_{\varepsilon}(x).$$
It is clear that $u_{\varepsilon}$ is Lipschitz continuous. Suppose that the result is true for $u_{\varepsilon}$, i.e. that $u_{\varepsilon}^*\in W^{1,p}(\R^d)$ and
\begin{equation}\label{lem7_eq1}
\|\nabla u_{\varepsilon}^*\|_p \leq \|\nabla u_{\varepsilon}\|_p\,,
\end{equation}
where 
\begin{equation*}
u_{\varepsilon}^*(x)  = \sup_{\tau > 0} u_{\varepsilon}*K_{\tau}(x) = \sup_{t > \varepsilon} u_0*K_t(x).
\end{equation*}
From Young's inequality we have 
\begin{equation}\label{lem7_eq2}
\|u_{\varepsilon}\|_p \leq \|u_0\|_p
\end{equation}
and, together with Minkowski's inequality, we also have
\begin{equation}\label{lem7_eq3}
\|\nabla u_{\varepsilon}\|_p \leq \|\nabla u_0\|_p
\end{equation}
for any $\varepsilon >0$. From \eqref{lem7_eq1}, \eqref{lem7_eq2} and \eqref{lem7_eq3} we find that $u_{\varepsilon}^*$ is uniformly bounded in $W^{1,p}(\R^d)$. Note that $u_{\varepsilon}^*$ converges pointwise to $u^*$ as $\epsilon \to 0$. From the weak compactness of $W^{1,p}(\R^d)$ we then conclude that $u^* \in W^{1,p}(\R^d)$ and $u_{\varepsilon}^* \rightharpoonup u^*$ as $\varepsilon \to 0$. Therefore 
$$\|\nabla u^*\|_p \leq \liminf_{\varepsilon \to 0} \|\nabla u_{\varepsilon}^*\|_p \leq  \liminf_{\varepsilon \to 0}  \|\nabla u_{\varepsilon}\|_p \leq \|\nabla u_0\|_p\,,$$
which completes the proof.
\end{proof}

\subsection{Proof of part (iv)} In the case $p=\infty$, we know that $u_0 \in W^{1,\infty}(\R^d)$ can be modified on a set of measure zero to become Lipschitz continuous with ${\rm Lip}(u_0) \leq \|\nabla u_0\|_{\infty}$. From Lemma \ref{lem4}, $u^*$ will also be bounded and Lipschitz continuous, with ${\rm Lip}(u^*) \leq {\rm Lip}(u_0)$, and the result follows, since in this case $u^* \in W^{1,\infty}(\R^d)$ with $\|\nabla u^*\|_{\infty} \leq {\rm Lip}(u^*)$. 

\medskip

If $p=2$, we are essentially done as well. In fact, from Lemma \ref{lem7} it suffices to consider the case $u_0$ Lipschitz continuous in $W^{1,2}(\R^d)$. In this case, from Lemma \ref{lem4} we know that $u^*$ is also Lipschitz continuous, and from Lemma \ref{lem5} we have that $u^*$ is subharmonic in the open set $A = \{x \in \R^d; \,u^*(x) > u_0(x)\}$. Recall from our discussion in the introduction of the paper that we already have $u^* \in W^{1,2}(\R^d)$, and thus the hypotheses of Lemma \ref{lem6} apply to $f = u^*$ and $g = (u^* - u_0)$. Therefore,
\begin{align*}
\|\nabla u_0\|_2^2 &= \int_{\R^d} |\nabla u_0|^2\,\dx = \int_{\R^d} |\nabla u^* - \nabla (u^* - u_0)|^2\,\dx\\
& = \int_{\R^d} |\nabla (u^*-u_0)|^2\,\dx - 2 \int_{\R^d} \nabla u^*\,.\,\nabla (u^*-u_0)\,\dx + \int_{\R^d} |\nabla u^*|^2\,\dx\\
& \geq \int_{\R^d} |\nabla u^*|^2\,\dx = \|\nabla u^*\|_2^2\,,
\end{align*}
which concludes the proof.

\subsection{Proof of part (i) - case $1< p < \infty$} 

\subsubsection*{Step 1: Set up} The initial considerations are the same as before. From Lemma \ref{lem7} it suffices to consider the case $u_0$ Lipschitz continuous in $W^{1,p}(\R)$. From Lemma \ref{lem4} we know that $u^*$ is also Lipschitz, and from Lemma \ref{lem5} we have that $u^*$ is subharmonic in the open set $A = \{x \in \R; \,u^*(x) > u_0(x)\}$. Let us explore the structure of $\R$ to write the open set $A$ as a countable union of disjoint open intervals
\begin{equation*}
A = \bigcup_{j} I_j =  \bigcup_{j} \, (\alpha_j, \beta_j).
\end{equation*}
Note also that subharmonicity is equivalent to convexity in each $(\alpha_j, \beta_j)$, when we deal with continuous functions.

\subsubsection*{Step 2: Zorn's lemma} As we did in the discrete case, we use here an argument based on Zorn's lemma. Define the family of Lipschitz continuous functions 
\begin{equation*}
\mc{S} = 
\left\{
\begin{array}{cc}
f:\R \to \R;\\
\\
u_0(x) \leq f(x) \leq u^*(x);& \ \ \forall x \in \R;\\
\\
{\rm Lip}(f) \leq {\rm Lip}(u_0);&\\
\\
\|f'\|_p \leq \|u_0'\|_p;&
\end{array}
\right.
\end{equation*}
The family $\mc{S}$ is non-empty since $u_0 \in \mc{S}$. We put a partial order $\preceq$ in $\mc{S}$ by considering the pointwise order (i.e. $f \preceq g$ in $\mc{S}$ if and only if $f(x) \leq g(x)$ for all $x \in \R$). Let us prove that $(\mc{S},\preceq)$ is inductive. Let $\{f_{\alpha}\}_{\alpha \in \Lambda}$ be a totally ordered subset and define
$$\overline{f}(x) = \sup_{\alpha  \in \Lambda}f_{\alpha}(x).$$
We claim that $\overline{f} \in \mc{S}$. Being a pointwise supremum of uniformly Lipschitz functions, we have ${\rm Lip}(\overline{f})\leq  {\rm Lip}(u_0)$. Let us write $L = {\rm Lip}(u_0)$. For each $N \in \N$ consider the $2N^2 +1$ points $\{j/N\}$ with $-N^2 \leq j \leq N^2$, $j \in \Z$. For each of these $j's$ choose a function $f_{j,N} \in \Lambda$ such that $\overline{f}(j/N) - f_{j,N}(j/N) < \frac{1}{N}$. Then choose 
$$f_N = \sup_{-N^2 \leq j \leq N^2} f_{j,N}.$$
It is clear that for each $x \in [-N,N]$, since $|x-j/N| \leq 1/2N$ for some $j$, we have
$$\overline{f}(x) - f_N(x) \leq \frac{1}{N} + \frac{L}{N}.$$
Therefore $f_N \to \overline{f}$ pointwise as $N\to \infty$. From the conditions on the family $\mc{S}$ we know that $\|f_N\|_{W^{1,p}}$ is uniformly bounded, and from the weak compactness of $W^{1,p}(\R)$ we must have $\overline{f} \in W^{1,p}(\R)$ and $f_N \rightharpoonup \overline{f}$. We then arrive at the bound
$$\big\|\big(\overline{f}\big)'\big\|_p \leq \liminf_{N\to \infty} \|f_N'\|_p \leq \|u_0'\|_p\,,$$
which shows that $\overline{f} \in \mc{S}$. From Zorn's lemma we guarantee the existence of (at least) one maximal element in $(\mc{S}, \preceq)$, which we call $g$.

\subsubsection*{Step 3: Finding an appropriate segment to cut} We want to show that $g = u^*$. Suppose this is not the case, i.e. that the open set $B = \{x \in \R; \,u^*(x) > g(x)\} \subset A$ is non-empty. Let us write $B$ as a countable union
\begin{equation*}
B= \bigcup_{l} Q_l =  \bigcup_{l} \, (\gamma_l, \delta_l).
\end{equation*}
We claim that $g$ cannot be superharmonic on $B$. In fact, if one of the intervals $(\gamma_l, \delta_l)$ is finite, since $u^*(\gamma_l) = g(\gamma_l)$ and $u^*(\delta_l) = g(\delta_l)$, the maximum principle would give us $u^* \equiv g$ in $[\gamma_l, \delta_l]$, a contradiction. If the interval is of type $(\gamma_l, \infty)$ (resp. $(-\infty, \delta_l)$), we would have $(u^* - g)$ strictly positive and convex in $(\gamma_l, \infty)$, with $(u^* - g)(\gamma_l)=0$. This is a contradiction since $(u^* - g)$ is Lipschitz and belongs to $L^p(\R)$, and thus must tend to zero at infinity. To conclude the proof of the claim, note that we cannot have $B = (-\infty,\infty)$, since $u_0(x_0) = g(x_0) = u^*(x_0)$ at the global maximum $x_0$ of $u_0$. 

\medskip

Therefore, there exists an interval $[a,b] \subset B$ such that 
\begin{equation}\label{pf_thm1_eq1}
g\left(\frac{a+b}{2}\right) < \frac{g(a) + g(b)}{2}.
\end{equation}
Let $\ell(x)$ be the equation of the line connecting the points $(a, g(a))$ and $(b,g(b))$, i.e.
$$l(x) = \frac{g(b) - g(a)}{b-a}(x-a) + g(a).$$
Let us consider the functions $\widetilde{u^*}(x) := u^*(x) - \ell(x)$ and  $\widetilde{g}(x) := g(x) - \ell(x)$. Let $y_0$ be the point of minimum of $\widetilde{g}$ when restricted to $[a,b]$. From \eqref{pf_thm1_eq1} we have $\widetilde{g}(y_0) \leq \widetilde{g}((a+b)/2) < 0$. We claim that there exists a line $\widetilde{\ell}$ parallel to the $x$-axis such that the graph of $\widetilde{u^*}$ is above $\widetilde{\ell}$ and the graph of $\widetilde{g}$ is below $\widetilde{\ell}$ in a neighborhood of $y_0$. To see this start by noting that $\widetilde{u^*}(y_0) - \widetilde{g}(y_0) = C >0$. For each $- \widetilde{g}(y_0) > \varepsilon >0$ let 
$$a_{\varepsilon} = \max\{a \leq x \leq y_0;\ \widetilde{g}(x) \geq \widetilde{g}(y_0) + \varepsilon\}$$
and
$$b_{\varepsilon} = \min\{y_0 \leq x \leq b;\ \widetilde{g}(x) \geq \widetilde{g}(y_0) + \varepsilon\}.$$
From this we have $\widetilde{g}(x) \leq \widetilde{g}(y_0) + \varepsilon$, for each $x \in [a_{\varepsilon}, b_{\varepsilon}]$, with equality on the endpoints. Suppose that for each $\varepsilon >0$ there exists a point $z_{\varepsilon} \in [a_{\varepsilon}, b_{\varepsilon}]$ such that $\widetilde{u^*}(z_{\varepsilon}) < \widetilde{g}(y_0) + \varepsilon$. There will be a subsequence of $\{z_{\varepsilon}\}_{\varepsilon>0}$ that accumulates around a certain $z_0 \in [a,b]$ thus giving 
$$\widetilde{u^*}(z_0) \leq \widetilde{g}(y_0) \leq \widetilde{g}(z_0) < \widetilde{u^*}(z_0),$$
a contradiction. Therefore, we can find an $\varepsilon>0$ such that 
$$\widetilde{u^*}(x) \geq \widetilde{g}(y_0) + \varepsilon$$ 
for each $x \in [a_{\varepsilon}, b_{\varepsilon}]$.

\medskip

If we undo the $\thicksim$ operation and return to the original picture, we have found a finite interval $[a_{\varepsilon}, b_{\varepsilon}]$ such that $g$ is below the line connecting $(a_{\varepsilon}, g(a_{\varepsilon}))$ to $(b_{\varepsilon}, g(b_{\varepsilon}))$ in $[a_{\varepsilon}, b_{\varepsilon}]$ (being strictly below in $(a_{\varepsilon}, b_{\varepsilon})$) and $u^*$ is above this line.

\subsubsection*{Step 4: Conclusion} Let us define a new function
\begin{equation*}
h(x) = \left\{
\begin{array}{lr}
g(x)\,,&{\rm if}\ x \notin [a_{\varepsilon}, b_{\varepsilon}]\\
\\
\displaystyle\frac{g(b_{\varepsilon}) - g(a_{\varepsilon})}{b_{\varepsilon}-a_{\varepsilon}}(x-a_{\varepsilon}) + g(a_{\varepsilon})\,,&{\rm if}\ x \in [a_{\varepsilon}, b_{\varepsilon}].
\end{array}
\right.
\end{equation*}
From the previous step we know that $u_0 \leq h \leq u^*$ and it is also clear that ${\rm Lip}(h) \leq {\rm Lip}(g) \leq {\rm Lip}(u_0)$. Finally, by Jensen's inequality we obtain
\begin{align*}
\|g'\|_p^p  & = \int_{[a_{\varepsilon}, b_{\varepsilon}]^c} |g'(x)|^p\,\dx + \int_{[a_{\varepsilon}, b_{\varepsilon}]} |g'(x)|^p\,\dx \\
& \geq  \int_{[a_{\varepsilon}, b_{\varepsilon}]^c} |g'(x)|^p\,\dx + (b_{\varepsilon}-a_{\varepsilon}) \left(\int_{a_{\varepsilon}}^{b_{\varepsilon}} |g'(x)| \frac{\dx}{(b_{\varepsilon}-a_{\varepsilon})}\right)^p\\
&  \geq \int_{[a_{\varepsilon}, b_{\varepsilon}]^c} |g'(x)|^p\,\dx + (b_{\varepsilon}-a_{\varepsilon}) \left|\int_{a_{\varepsilon}}^{b_{\varepsilon}} g'(x) \frac{\dx}{(b_{\varepsilon}-a_{\varepsilon})}\right|^p\\
& = \int_{[a_{\varepsilon}, b_{\varepsilon}]^c} |g'(x)|^p\,\dx + (b_{\varepsilon}-a_{\varepsilon}) \left| \frac{g(b_{\varepsilon}) - g(a_{\varepsilon})}{(b_{\varepsilon}-a_{\varepsilon})}\right|^p\\
& = \int_{\R} |h'(x)|^p\,\dx = \|h'\|_p^p.
\end{align*}
Therefore we get that $h \in \mc{S}$ but this is a contradiction since $h$ is strictly bigger than the maximal element $g$ in $(a_{\varepsilon}, b_{\varepsilon})$. This shows that $g = u^*$ and the proof is concluded. 

\subsection{Proof of part (ii)} The argument we shall use for this part is inspired in Tanaka's \cite{Ta}. Recall that when $u_0 \in W^{1,1}(\R)$, after adjusting on a set of measure zero, $u_0$ may be taken to be absolutely continuous. From Lemma \ref{lem4}  we see that $u^*$ is also continuous and the set $A = \{x \in \R; \ u^*(x) > u_0(x)\}$ is open. Let us again write $A$ as a countable union of disjoint open intervals
\begin{equation*}
A = \bigcup_{j} I_j =  \bigcup_{j} \, (\alpha_j, \beta_j).
\end{equation*}
From Lemma \ref{lem5} we know that $u^*$ is subharmonic (thus convex) in each subinterval $I_j = (\alpha_j, \beta_j)$. Therefore, $u^*$ must be locally Lipschitz on each $I_j$ and, in particular, it is absolutely continuous on each compact subinterval of $I_j$. From this we conclude that $u^*$ is differentiable a.e. on each $I_j$, with derivative that we will denote by $v$. 

\smallskip

We observe now that in each subinterval $I_j$ the variation of $u^*$ is smaller than the variation of $u_0$. In fact, since $u^*$ is convex on $I_j$,  let  $\gamma_j \in [\alpha_j, \beta_j]$ be a minimum of $u^*$ on $[\alpha_j, \beta_j]$ (note that we might have $\gamma_j = \alpha_j$ or $\gamma_j = \beta_j$). The crucial observation is that $u^*$ is monotone in $[\alpha_j, \gamma_j]$ and in $[\gamma_j, \beta_j]$, thus leading to (using the continuity of $u^*$ and approaching by compacts from inside $I_j$)
\begin{align}\label{part(ii)_eq3}
\begin{split}
\int_{I_j} |v(x)|\,\dx& = \big[u^*(\alpha_j) - u^*(\gamma_j)\big] + \big[u^*(\beta_j) - u^*(\gamma_j)\big]\\
& \leq \big[u_0(\alpha_j) - u_0(\gamma_j)\big] + \big[u_0(\beta_j) - u_0(\gamma_j)\big]\\
& \leq \int_{\alpha_j}^{\gamma_j} |u_0'(x)|\,\dx + \int_{\gamma_j}^{\beta_j} |u_0'(x)|\,\dx\\
& = \int_{I_j} |u_0'(x)|\,\dx.
\end{split}
\end{align}
Note that in case $\alpha_j = -\infty$ (resp. $\beta_j = \infty$) we have $u_0(\alpha_j) = 0$ and $u^*(\alpha_j) =0$ (resp. $u_0(\beta_j) = 0$ and $u^*(\beta_j) = 0$) due to the fact that $u_0 \in W^{1,1}(\R)$ and $u^*\in L^1_{weak}(\R)$ and is convex on $I_j$. In particular, since $u_0 \in W^{1,1}(\R)$, we have that $v \in L^1(A)$. 

\smallskip

We shall prove now that $u^*$ is weakly differentiable with 
\begin{equation}\label{part(ii)_eq2}
(u^*)' = \chi_{A^c} \,u_0' + \chi_{A} \,v,
\end{equation}
where $\chi_{A}$ and $\chi_{A^c}$ denote the indicator functions of the sets $A$ and $A^c$. In fact, let $\varphi \in C^{\infty}_c(\R)$. Observe first that 
\begin{equation}\label{part(ii)_eq1}
\int_{I_j} u^*(x)\,\varphi'(x)\,\dx = \big[u_0(\beta_j)\varphi(\beta_j)- u_0(\alpha_j)\varphi(\alpha_j)\big] - \int_{I_j} v(x)\, \varphi(x)\,\dx,
\end{equation}
obtained again by the continuity of $u^*$ and a limiting argument approaching by compacts from inside $I_j$.  From \eqref{part(ii)_eq1} we have
\begin{align*}
\begin{split}
&\int_{\R} u^*(x)\,\varphi'(x)\,\dx  = \int_{A^c} u^*(x)\,\varphi'(x)\,\dx +  \int_{A} u^*(x)\,\varphi'(x)\,\dx\\
& = \int_{A^c} u_0(x)\,\varphi'(x)\,\dx + \sum_j \big[u_0(\beta_j)\varphi(\beta_j)- u_0(\alpha_j)\varphi(\alpha_j)\big]  - \int_{A} v(x)\, \varphi(x)\,\dx\\
& = \int_{A^c} u_0(x)\,\varphi'(x)\,\dx + \left[\int_{A} u_0(x)\,\varphi'(x)\,\dx + \int_{A} u_0'(x)\,\varphi(x)\,\dx\right] - \int_{A} v(x)\, \varphi(x)\,\dx\\
& =  -\int_{\R} u_0'(x)\,\varphi(x)\,\dx + \int_{A} u_0'(x)\,\varphi(x)\,\dx- \int_{A} v(x)\, \varphi(x)\,\dx\\
& = - \int_{\R} (\chi_{A^c} \,u_0' + \chi_{A} \,v)(x)\,\varphi(x)\,\dx,
\end{split}
\end{align*}
as we wanted to show. 

\smallskip

We are now in position to conclude. From \eqref{part(ii)_eq3} and \eqref{part(ii)_eq2} we have
\begin{align*}
\|(u^*)'\|_1  & = \int_{\R} |(u^*)'(x)|\,\dx \\
& = \int_{A^c} |(u_0)'(x)|\,\dx + \int_{A} |v(x)|\,\dx\\
& \leq  \int_{A^c} |(u_0)'(x)|\,\dx + \int_{A} |(u_0)'(x)|\,\dx = \|(u_0)'\|_1.
\end{align*}

\subsection{Proof of part (iii)} If $u_0$ has bounded variation, then $u_0$ is bounded. The distributional derivative $Du_0$ is a Radon measure with $|Du_0| \leq V(u_0)$, where $|Du_0|$ denotes the total variation of $Du_0$. For $\varepsilon>0$ we consider again
$$u_{\varepsilon}(x) = u_0*K_{\varepsilon}(x).$$
Note that $u_{\varepsilon} \in C^{\infty}(\R)$ is bounded and Lipschitz. We let 
\begin{equation*}
u_{\varepsilon}^*(x)  = \sup_{\tau > 0} u_{\varepsilon}*K_{\tau}(x) = \sup_{t > \varepsilon} u_0*K_t(x).
\end{equation*}
From Lemma \ref{lem5} we know that $u_{\varepsilon}^*$ is subharmonic in the open set A = $\{x \in \R;\ u_{\varepsilon}^*(x) > u_{\varepsilon}(x)\}$. Let us write again
\begin{equation*}
A = \bigcup_{j} I_j =  \bigcup_{j} \, (\alpha_j, \beta_j),
\end{equation*}
and thus we will have $u_{\varepsilon}^*$ convex in each $I_j = (\alpha_j, \beta_j)$. Now consider a partition $\mc{P} = \{x_1,x_2,...,x_N\}$. Refine this partition by including the endpoints $\alpha_j$ and $\beta_j$ for which $x_k \in [\alpha_j, \beta_j]$ for $k=1,2,...,N$. Thus we obtain a new partition $\mc{P'} = \{y_1,...,y_M\} \supset \mc{P},$ where $M\geq N$. We now estimate the variation $V_{\mc{P}}(u_{\varepsilon}^*)$ of the function $u_{\varepsilon}^*$ with respect to the partition $\mc{P}$ by observing that
\begin{align}\label{part(iii)_eq1}
\begin{split}
V_{\mc{P}}(u_{\varepsilon}^*) & = \sum_{i=1}^{N-1} \big| u_{\varepsilon}^*(x_{i+1}) - u_{\varepsilon}^*(x_{i})\big| \\
& \leq V_{\mc{P'}}(u_{\varepsilon}^*) = \sum_{j=1}^{M-1} \big| u_{\varepsilon}^*(y_{j+1}) - u_{\varepsilon}^*(y_{j})\big|\\
& \leq V(u_{\varepsilon})\,,
\end{split}
\end{align}
where in the last inequality we used the convexity of $u_{\varepsilon}^*$ in each $I_j = (\alpha_j, \beta_j)$ and the fact that $u_{\varepsilon}^*$ and $u_{\varepsilon}$ agree at the endpoints $\alpha_j$ and  $\beta_j$ (minor modifications are needed for the cases $\alpha_j = -\infty$ or $\beta_j = \infty$). Since
$$Du_{\varepsilon}(x) = Du_0*K_{\varepsilon}(x),$$
from Young's inequality we have
\begin{equation}\label{part(iii)_eq2}
V(u_{\varepsilon}) =  |Du_{\varepsilon}| \leq |Du_0|\,\|K_{\varepsilon}\|_1 = |Du_0| \leq V(u_0).
\end{equation}

\smallskip

We now observe that as $\varepsilon \to 0$, we have $u_{\varepsilon}^* \to u^*$ pointwise. From \eqref{part(iii)_eq1} and \eqref{part(iii)_eq2} we find that 
\begin{align*}
V_{\mc{P}}(u^*) & = \sum_{i=1}^{N-1} \big| u^*(x_{i+1}) - u^*(x_{i})\big|\\
& = \lim_{\varepsilon \to 0} \sum_{i=1}^{N-1} \big| u_{\varepsilon}^*(x_{i+1}) - u_{\varepsilon}^*(x_{i})\big| \leq V(u_0).
\end{align*}
Since this holds for any partition $\mc{P}$ we obtain
$$V(u^*) \leq V(u_0),$$
and the proof is concluded.

%%%%%%%%%%%%%%%%%%%%%%%%%%%%%%%%%%%%%%%%%%%%%%%%%%%%%%%%%%%%%%%%%%%%%%%%%%%%%%%%%%%%%%%%%%%%%%%%%%%%%%%%%%%%%%%%%%%%%%%%%%%

\section{Proof of Theorem \ref{thm_Poisson} - Continuous Poisson kernel}
Throughout this proof we will again assume without loss of generality that $u_0 \geq 0$. 

\subsection{Preliminaries} The first result of this section is analogous to Lemma \ref{lem4}.

\begin{lemma}\label{lem9}.
\begin{itemize}
\item[(i)] If $u_0 \in C(\R) \cap L^p(\R^d)$, for some $1\leq p < \infty$, then $u^{\star} \in C(\R^d)$. 
\smallskip
\item[(ii)] If $u_0$ is bounded and Lipschitz continuous then $u^{\star}$ is bounded and Lipschitz continuous with ${\rm Lip}(u^{\star}) \leq {\rm Lip}(u_0)$.
\end{itemize}
\end{lemma}
\begin{proof}
Just follow the proof of Lemma \ref{lem4}.
\end{proof}
We now investigate the set where $u^{\star}$ disconnects from $u_0$. As in the heat kernel case, we will show that $u^{\star}$ is subharmonic in this set. The main tool for this will be the structure of the underlying Laplace equation (namely, the mean value property).

\begin{lemma}[Subharmonicity II]\label{lem10} Let $u_0 \in C(\R) \cap L^p(\R^d)$ for some $1\leq p < \infty$ or $u_0$ be bounded and Lipschitz continuous. Then $u^{\star}$ is subharmonic in the open set $A = \{x \in \R^d; \,u^{\star}(x) > u_0(x)\}$.
\end{lemma}
\begin{proof}
From Lemma \ref{lem9} we see that $u^{\star}$ is continuous and $A$ is in fact an open set. In what follows we keep denoting by $B_r(x)$ the open $d$-dimensional ball centered in $x$ with radius $r$, and now we introduce $B_r(x,y)$ to denote the open $(d+1)$-dimensional ball centered in $(x,y)$ with radius $r$.

\medskip

Let $x_0 \in A$. Since $u_0(x_0) < u^{\star}(x_0)$ and
$$\lim_{y \to 0^+} u(x_0,y) = u_0(x_0)\,,$$
there exists $\delta = \delta(x_0) >0$ such that if $y <\delta$ then 
\begin{equation}\label{Lem10_eq0}
u(x_0,y) < u^{\star}(x_0) - \tfrac{1}{2} (u^{\star}(x_0) - u_0(x_0)).
\end{equation}
Now let $y_0 \geq \delta(x_0)$. Choose a radius $0< r_0 < \delta $ such that $\overline{B_{r_0}(x_0)} \subset A$. For any $r < r_0$ we have $\overline{B_{r}(x_0,y_0)} \subset A \times (0,\infty)$. Recall that the function $u(x,y)$ is harmonic in $\R^d \times (0,\infty)$ and therefore, by the mean value property, we have

\begin{equation}\label{Lem10_eq1}
u(x_0,y_0) = \frac{1}{r^{d+1}\omega_{d+1}} \int_{B_{r}(x_0,y_0)} u(x,y)\,\dx\,\dy \,,
\end{equation}
where $\omega_{d+1}$ denotes the volume of the $(d+1)$-dimensional unit ball. From \eqref{Lem10_eq1} we arrive at
\begin{align}\label{Lem10_eq2}
\begin{split}
u(x_0,y_0) & \leq \frac{1}{r^{d+1}\omega_{d+1}} \int_{B_{r}(x_0,y_0)} u^{\star}(x)\,\dx\,\dy\\
& = \frac{1}{r^{d+1}\omega_{d+1}} \int_{B_{r}(x_0)} 2\, \sqrt{r^2 - |x-x_0|^2}\,\, u^{\star}(x)\,\dx.
\end{split}
\end{align}
By \eqref{Lem10_eq0} we know that $u^{\star}(x_0) = \sup_{y \geq \delta} u(x_0,y)$ and since \eqref{Lem10_eq2} holds for any $y_0 \geq \delta$, we have 
\begin{equation}\label{Lem10_eq3}
u^{\star}(x_0) \leq \frac{1}{r^{d+1}\omega_{d+1}} \int_{B_{r}(x_0)} 2\, \sqrt{r^2 - |x-x_0|^2}\,\, u^{\star}(x)\,\dx
\end{equation}
for any $r < r_0 = r_0(x_0)$. 

\medskip

Condition \eqref{Lem10_eq3} can be viewed as a ``weighted" subharmonicity. We shall prove that it actually implies subharmonicity in the usual sense. First observe that \eqref{Lem10_eq3} is sufficient to establish the maximum principle for $u^{\star}$ in the domain $A$ (in each connected component to be precise), i.e. if $\Omega$ is an open connected set such that $\overline{\Omega} \subset A$, and $u^{\star}$ has an interior maximum in $\overline{\Omega}$ then $u^{\star}$ must be constant in $\overline{\Omega}$. With this in hand, consider any ball $\overline{B_s(x_0)} \subset A$ and let $h:\overline{B_s(x_0)} \to \R$ be the solution of the Dirichlet boundary value problem
\begin{equation*}
\left\{
\begin{array}{rl}
\Delta h = 0& \ {\rm in} \ B_s(x_0);\\
h = u^{\star}& \ {\rm in} \ \partial B_s(x_0).
\end{array}
\right.
\end{equation*}
Consider the function $g= u^{\star} - h$. Let us prove that $g$ satisfies the same local ``weighted" subharmonicity \eqref{Lem10_eq3} in $B_s(x_0)$. In fact, for a given $x_1 \in B_s(x_0)$, we know that \eqref{Lem10_eq3} holds in a neighborhood of $x_1$. Therefore, we can find a radius $r_1$ such that $\overline {B_{r_1}(x_1)} \subset B_s(x_0)$ and \eqref{Lem10_eq3} holds for $r < r_1$. Using the fact that $h$ is harmonic and that our weight is a radial function we arrive at
\begin{align*}
g(x_1) & = u^{\star}(x_1) - h(x_1)\\
& \leq  \left(\frac{1}{r^{d+1}\omega_{d+1}} \int_{B_{r}(x_1)} 2\, \sqrt{r^2 - |x-x_1|^2}\,\, u^{\star}(x)\,\dx\right)  - h(x_1)\\
& = \left(\frac{1}{r^{d+1}\omega_{d+1}} \int_{B_{r}(x_1)} 2\, \sqrt{r^2 - |x-x_1|^2}\,\, \big\{ u^{\star}(x) - h(x)\big\} \,\dx\right) \\
& = \left(\frac{1}{r^{d+1}\omega_{d+1}} \int_{B_{r}(x_1)} 2\, \sqrt{r^2 - |x-x_1|^2}\,\, g(x) \,\dx\right).
\end{align*}
By the maximum principle, since $g$ is continuous in $\overline{B_s(x_0)}$, the maximum of $g$ in $\overline{B_s(x_0)}$ must be attained on the boundary. However, $g = 0$ in $\partial B_s(x_0)$, and therefore 
$$u^{\star}(x_0) \leq h(x_0),$$
which shows that $u^{\star}$ is subharmonic (in the usual sense, defined after Lemma \ref{lem4}) since $h$ is harmonic and thus equal to its average over the sphere $\partial B_s(x_0)$, where $h = u^{\star}$ by construction.
\end{proof}

\subsection{Proof of Theorem \ref{thm_Poisson}} Having established the subharmonicity of $u^{\star}$ in the disconnecting set (Lemma \ref{lem10}), the proof of Theorem \ref{thm_Poisson} follows essentially in the same way as the proof of Theorem \ref{thm_cont}, by observing that the Poisson kernel also satisfies the semigroup property $P_{y_1} * P_{y_2} = P_{y_1 + y_2}$ to reduce to the Lipschitz case as in Lemma \ref{lem7}. We shall omit the details.

\section{Acknowledgements}
We would like to thank the following colleagues for helpful remarks during the preparation of this manuscript: William Beckner, Luis Caffarelli, Kevin Hughes, Milton Jara, Diego Moreira, Roberto Oliveira, Lillian Pierce, Luis Silvestre and Ralph Teixeira. E. C. acknowledges support from CNPq-Brazil grants $473152/2011-8$ and $302809/2011-2$, and FAPERJ grant $E-26/103.010/2012$. B. S. acknowledges support from CNPq-Brazil grants $302962/2011-5$, $474944/2010-7$, FAPERJ grant $E-26/102.940/2011$ and PRONEX Optimization.

\end{document}